\documentclass[12pt]{amsart}
\usepackage{hyperref}
\hypersetup{nesting=true,debug=true,naturalnames=true}
\usepackage{graphicx,amssymb}

\usepackage[square,sort&compress,comma,numbers]{natbib}
\usepackage{nicefrac,xcolor,upref,version}

\usepackage{amscd,amsthm,amsmath,amssymb,amsfonts}

\usepackage{mathrsfs}

\newtheorem{theorem}{Theorem}[section]
\newtheorem{lemma}[theorem]{Lemma} 
\newtheorem{proposition}[theorem]{Proposition} 
\newtheorem{definition}[theorem]{Definition} 
\newtheorem{problem}[theorem]{Problem}

\numberwithin{equation}{section}

\def\astmult{{\times}}

\def\tmu{{\mu^{\times}}}
\def\htmu{{\widehat{\mu}^{\times}}}
\def\heta{{\widehat {\eta}}}

\def\ord{{\rm ord}}
\def\Q{{\mathbb {Q}}}
\def\Z{{\mathbb Z}}  
 
\def\eps{{\varepsilon}}

\def\bff{{\bf f}}
\def\bft{{\bf t}}

\def\xifi{{\xi_{{\bf f}, p}}}

\def\xitm{{\xi_{{\bf t}, p}}}
\def\txitm{{{\widetilde \xi}_{{\bf t}, p}}}

\def\xig{\Xi}

\def\tig{{\widetilde g}}
\def\cL{{\mathcal L}}

\def\tiq{{\tilde q}}
\def\tip{{\tilde p}}

\def\tir{{\widetilde r}}
\def\tix{{\widetilde x}}

\def\teta{{\widetilde \eta}}

\def\bfc{{\bf c}}

\def\ta{{\tilde a}}
\def\tb{{\tilde b}}
\def\tQ{{\tilde Q}}

\begin{document}

\title[Approximation to $p$-adic Thue--Morse numbers]{On the rational approximation 
to $p$-adic Thue--Morse numbers}

\author{Yann Bugeaud}
\address{Universit\'e de Strasbourg, Math\'ematiques,
7, rue Ren\'e Descartes, 67084 Strasbourg  (France)} 
\address{Institut universitaire de France} 
\email{bugeaud@math.unistra.fr}

\begin{abstract}
Let $p$ be a prime number and $\xi$ an irrational $p$-adic number. 
Its multiplicative irrationality exponent $\tmu (\xi)$ is the 
supremum of the real numbers $\tmu$ for which the inequality
$$
|b \xi - a|_{p} \leq  | a b |^{-\tmu / 2}  
$$
has infinitely many solutions in nonzero integers $a, b$.  
We show that $\tmu (\xi)$ can be expressed in terms of a new exponent 
of approximation attached to a sequence of rational numbers defined in terms of $\xi$. 
We establish that $\tmu (\xitm) = 3$, where $\xitm$ is the $p$-adic number 
$1 - p - p^2 + p^3 - p^4 + \ldots$, whose sequence of digits is given by 
the Thue--Morse sequence over $\{-1, 1\}$. 
\end{abstract}


\maketitle

\section{Introduction}\label{sec:1}

Throughout this paper, we let $p$ denote a prime number.
Let $\xi$ be an irrational $p$-adic number. The irrationality exponent $\mu (\xi)$ of $\xi$ is the 
supremum of the real numbers $\mu$ for which
\begin{equation}  \label{eq:irr1}
0 < |b \xi - a|_{p} \leq \max\{ |a|, |b| \}^{-\mu}    
\end{equation}   
has infinitely many solutions in nonzero integers $a, b$. 
As already pointed out in \cite{BdM19,BaBu21,BuSc21}, 
unlike in the real case, the integers $|a|$ and $|b|$ in \eqref{eq:irr1} 
do not necessarily have the same order of magnitude, and one of them can be 
much larger than the other one. 
This has motivated the study in \cite{BuSc21} of the following two 
exponents of $p$-adic multiplicative rational approximation.

\begin{definition}  \label{def1}
Let $\xi$ be an irrational $p$-adic number.  
The multiplicative irrationality exponent $\tmu (\xi)$ of $\xi$ is the 
supremum of the real numbers $\tmu$ for which
\begin{equation}  \label{eq:irr2}
0 < |b \xi - a|_{p} \leq  \bigl(| a b |^{1/2} \bigr)^{-\tmu}    
\end{equation}   
has infinitely many solutions in integers $a, b$. 
The uniform multiplicative irrationality exponent $\htmu (\xi)$ of $\xi$ is the 
supremum of the real numbers $\htmu$ for which the system
\begin{equation}  \label{eq:irr4}
0 < |a b|^{1/2} \le X, \quad |b \xi - a|_{p} \leq X^{-\htmu}    
\end{equation}   
has a solution in integers $a, b$ for every sufficiently large real number $X$. 
\end{definition}

We point out that $a$ and $b$ are not assumed to be coprime in \eqref{eq:irr1}, \eqref{eq:irr2}, 
nor in \eqref{eq:irr4}. Adding this assumption would not change the values of 
$\mu (\xi)$ and $\tmu (\xi)$, but would change the value 
of the uniform exponent $\htmu$ at some $p$-adic numbers $\xi$. 

It follows from the Minkowski Theorem \cite{Mah34a,Mah34b} 
and the obvious inequality 
$\max\{|a|, |b|\} \le |ab|$ valid for all nonzero integers $a, b$ that we have
\begin{equation}  \label{eq:ineq}
2 \le \mu(\xi) \le \tmu (\xi) \le 2 \mu (\xi). 
\end{equation}   
Inequalities \eqref{eq:ineq} are best possible; see \cite{BuSc21}. 
Furthermore, \cite[Theorem 3.1]{BuSc21} asserts that 
\begin{equation}  \label{eq:ineq2}
2  \le \htmu (\xi) \le  \frac{5 + \sqrt{5}}{2}, 
\end{equation}   
for every irrational $p$-adic number $\xi$, while there exist 
$p$-adic numbers $\xi$ with $\htmu (\xi) = 3$, an example 
being given by the $p$-adic Liouville number $\sum_{j \ge 1} p^{j!}$; see \cite{BuSc21} 
for additional results.

It is readily verified that $\tmu (\xi) = \tmu (u \xi / v)$ and $\htmu (\xi) = \htmu (u \xi / v)$ hold 
for every irrational $p$-adic number $\xi$ and every nonzero rational number $u/v$. 
However, the exponents $\tmu$ and $\htmu$ are not invariant by rational translations. 
To see this, observe that $|b \xi - a|_p = |b (\xi + 1) - (a + b)|_p$, while the product 
$|b (a + b)|$ is much larger than the product $|a b|$ when $|b|$ exceeds $|a|$. 
Consequently, $\tmu (\xi + 1)$ may be strictly smaller than $\tmu (\xi)$. 

The purpose of the present paper is to establish a somehow unexpected link between $p$-adic 
multiplicative approximation and sequences of continued fractions of rational numbers. 
The determination of $\tmu (\xi)$ and $\htmu (\xi)$ then boils down to the study of the 
size of the partial quotients of an infinite sequence of rational numbers. 
As an example of application, we determine the exact values of $\tmu (\xitm)$ and $\htmu (\xitm)$, where 
$\xitm$ is the $p$-adic Thue--Morse number over $\{-1, 1\}$, and lower bounds 
for $\tmu (\xifi)$ and $\htmu (\xifi)$, where $\xifi$ is the $p$-adic Fibonacci number over $\{0, 1\}$.

\section{Approximation to $p$-adic Thue--Morse numbers} \label{sec:2}

The Thue--Morse sequence $(t_n)_{n \ge 0}$ over $\{-1, 1\}$ is defined by 
$t_0 = 1$ and the recursion $t_{2n} = t_n, t_{2n+1} = - t_n$ for $n \ge 0$. 
The first letters of the Thue--Morse infinite word $\bft = t_0 t_1 t_2 \ldots$ 
are then $1$, $-1$, $-1$, $1$, $-1, \ldots$
Said differently, $\bft$ is the fixed point starting by $1$ of the substitution $\tau$ defined by $\tau (1) = 1 -1$
and $\tau(-1) =  -1 1$. 
Diophantine properties of real numbers whose $b$-ary expansion, for some integer $b \ge 2$, is 
given by the Thue--Morse sequence over some alphabet have been investigated 
in \cite{Bu11,BuQu13,BaZo15,BaZo20}. 
In the present work, we study multiplicative rational 
approximation to the $p$-adic Thue--Morse number $\xitm$ defined by
$$
\xitm = \sum_{n \ge 0} t_n p^n = 1 - p - p^2 + p^3 - p^4 + p^5 + p^6 - p^7 + \ldots
$$
Since $\bft$ is not ultimately periodic, 
$\xitm$ is an irrational number. 
It has been established in \cite{BuYao} that $\mu (\xitm) = 2$. 
We complement this result as follows. 

\begin{theorem} \label{th:main}
The $p$-adic Thue--Morse number $\xitm$ satisfies
$$
\tmu(\xitm) = 3 \quad \hbox{and} \quad \htmu(\xitm) = 2.
$$
More precisely, there exist positive real numbers $c_1, c_2$ with the 
following properties. There exist nonzero
integers $a, b$ with $|ab|$ arbitrarily large and 
\begin{equation} \label{good}
|b \xitm - a|_p < c_1 |ab|^{-3/2},       
\end{equation}
while, for every nonzero integers $r, s$, we have
\begin{equation} \label{notgood}
|s \xitm - r|_p > c_2 |rs|^{-3/2}.    
\end{equation}
\end{theorem}

Admittedly, it would be more natural to consider the $p$-adic number $\txitm$ whose Hensel expansion is
given by the Thue--Morse sequence written over 
$\{0, 1\}$, namely the $p$-adic number
$$
\txitm = \sum_{n \ge 0} t_n p^n = 1 + p^3 + p^5 + p^6 + \ldots = \frac{\xitm}{2}  + \frac{1}{2 (1 - p)}.
$$
It is very likely that the conclusion of Theorem \ref{th:main} holds for $\txitm$ and, more generally, 
for every $p$-adic number of the form $u \xitm + v$, where $u, v$ are rational numbers with 
$u$ nonzero. The difficulty lies in the 
control of $|b (u \xitm + v) - a|_p$ when $|b|$ exceeds $|a|$; see at the end of Section \ref{sec:7} 
for a short discussion. 
In principle, the same method 
based on Hankel determinants and Pad\'e approximants could be used, but new 
non-vanishing results for Hankel determinants are needed.

The strategy of the proof is the following. 
The easiest part, done in Section \ref{sec:4}, consists in using repetitions in the Thue--Morse 
word to exhibit an infinite family of integer pairs $(a, b)$ realizing \eqref{good}. 
Then, to prove that there are no better approximations, 
we apply an idea of Mahler 
to associate with every solution to 
$$
|b \xitm - a|_p < \frac{1}{2 |ab|} 
$$
an integer $m$ and a large partial quotient of the rational number
\begin{equation} \label{defz}
z_m = z_{m, p} = {t_{m-1} \over p} + {t_{m-2} \over p^2} + \ldots + {t_0 \over p^m}.     
\end{equation}
This allows us to transform a Diophantine question on $p$-adic numbers into a Diophantine 
question on a sequence of rational numbers. 
As a consequence, in order to establish \eqref{notgood}, 
it is sufficient to prove that no $z_m$ 
has a partial quotient exceeding some absolute constant times $p^{m/3}$; 
see Proposition \ref{precis}. 
Furthermore, we derive from an easy relation between the rational numbers $z_m$ and $z_n$, for 
distinct integers $m, n$, that the quotient of the largest 
partial quotient of $z_m$ by that of $z_n$ is, roughly speaking, bounded from above 
by $p^{|m-n|}$ and from below by $p^{-|m-n|}$. Consequently, 
it is sufficient for our purpose to prove that, for every large integer $k$, 
neither $z_{2^k}$, nor $z_{3 \cdot 2^k}$ 
have very large partial quotients (apart, in the case of $z_{3 \cdot 2^k}$, from the one coming from a
solution to \eqref{good}). 
Very good bounds for the partial quotients of $z_{2^k}$ have been obtained in \cite{BuHan21}; 
see Theorem \ref{th:BuHan} below. Regarding $z_{3 \cdot 2^k}$, we use a similar argument as 
in \cite{Bu11} to get in Section \ref{sec:5} an upper bound 
for its partial quotients strong enough for our purpose. 
The proof of Theorem \ref{th:main} is completed in Section \ref{sec:6}. 
We discuss the case of the $p$-adic Fibonacci number in Section \ref{sec:7}, while the last section is 
devoted to additional comments and alternative proofs of some results of \cite{BuSc21}. 

Throughout this text, the constants implied by $\ll$ and $\asymp$ are positive and absolute, 
and those implied by $\ll_c$ and $\asymp_c$ are positive and depend at most on the 
parameter $c$.

\section{Diophantine exponents associated with sequences of rational numbers}  \label{sec:3}

In this section, we use an idea of Mahler~\cite[pp. 64--67]{Mah61} to make a link between 
the multiplicative Diophantine exponents of a $p$-adic number $\xi$ and new exponents 
of approximation associated with an infinite sequence of rational numbers defined by means 
of the Hensel expansion of $\xi$ or some other expansion $c_0 + c_1 p + c_2 p^2 + \ldots$ of $\xi$, 
where $(c_k)_{k \ge 0}$ is a bounded sequence of integers.

\begin{definition} \label{defeta}
Let $b \ge 2$ be an integer. 
Let $\bfc = (c_k)_{k \ge 0}$ be a bounded sequence of integers 
and define
$$
x_m = \frac{c_0}{b^m} + \frac{c_1}{b^{m-1}} + \ldots  + \frac{c_{m-1}}{b}, \quad m \ge 1.
$$
We denote by $\eta_b (\bfc)$ (resp., $\heta_b (\bfc)$) the supremum of the real numbers $\eta$
such that, for arbitrarily large $m$ (resp., for every $m$ large enough), the rational number $x_m$ 
has a partial quotient greater than $b^{\eta m}$. 
Said differently, if for $m \ge 1$ we let $\eta_m$ denote the real number 
such that the largest partial quotient of 
$x_m$ is equal to $b^{\eta_m m}$, then we have
$$
\eta_b (\bfc) = \limsup_{m \to + \infty} \, \eta_m, \quad 
\heta_b (\bfc) = \liminf_{m \to + \infty} \, \eta_m. 
$$
\end{definition}

The exponents $\eta_b$ and $\heta_b$ take their values in the interval $[0, 1]$. 
They can also be defined for unbounded sequences $\bfc$, but for
simplicity we do not discuss this case. 
In the sequel, we only consider the case where $b$ is a prime number and,
when there is no ambiguity, we simply write $\eta$ and $\heta$, without subscript.

Throughout the rest of this section, $p$ denotes a prime number and we let $\bfc = (c_k)_{k \ge 0}$
and $(x_m)_{m \ge 1}$ be as in Definition \ref{defeta} with $b=p$. 

We begin with an analysis of the evolution of the partial quotients of the rational numbers 
$x_m$. We refer to \cite{BuLiv,SchmLN} for an introduction to the theory of continued fractions 
and use classical results without further notice. 

\begin{definition}   \label{assoc} 
Let $[a_0; a_1, a_2, \ldots]$ be a real number. 
For a positive integer $m$, we say that the partial quotient $a_m$ and the 
convergent $[a_0; a_1, a_2, \ldots , a_{m-1}]$ are associated. 
\end{definition}

Let $m \ge 2$ be an integer and 
$a/b$ a convergent of $x_m$ associated with the partial quotient $r$. 
Assume that  $r \ge 2p$. 
It follows from the theory of continued fractions that 
$$
\frac{1}{(r+2) b^2} < \Big| x_m - {a \over b} \Big| < \frac{1}{r b^2}. 
$$
Since $p x_m = x_{m-1} + c_{m-1}$, we get
\begin{equation}\label{ineq}
\frac{p}{(r + 2) b^2} < \Big| x_{m-1} - {p a - b c_{m-1} \over b} \Big| = p \Big| x_m - {a \over b} \Big|
< \frac{p}{r b^2} \le \frac{1}{2 b^2}.
\end{equation}
By Legendre's theorem, the rational number $(p a - b c_{m-1}) / b$ is a convergent of $x_{m-1}$. 
If $p$ does not divide $b$, then it is written under its reduced form and 
is associated with a partial quotient $r'$ satisfying 
$$
\frac{1}{(r' + 2) b^2} < \Big| x_{m-1} - {p a - b c_{m-1} \over b} \Big| < \frac{1}{r' b^2},
$$
thus,
$r / p - 2 < r' < (r+2) / p$, by \eqref{ineq}. 
If $p$ divides $b$, then $(a - b c_{m-1} / p) / (b/p)$ is a convergent of $x_{m-1}$ 
written under its reduced form and 
associated with a partial quotient $r'$ satisfying 
$$
\frac{1}{(r' + 2) (b/p)^2} < \Big| x_{m-1} - {a - b c_{m-1} /p  \over (b/p)} \Big| < \frac{1}{r' (b/p)^2},
$$
thus, $p r - 2 < r' < p (r+2)$, again by \eqref{ineq}.  

Likewise, since $x_{m+1} = (x_{m} + c_{m}) / p$, we get
$$
\frac{1}{(r + 2) p b^2} < \Big| x_{m+1} - {a + b c_{m} \over b p} \Big| 
= {1 \over p} \,  \Big| x_m - {a \over b} \Big| < \frac{1}{r p b^2} \le \frac{1}{2 (p b)^2}.
$$
By Legendre's theorem, the rational number $(a +  b c_{m}) / (b p)$ is a convergent of $x_{m+1}$. 
If $p$ does not divide $a + b c_m$, then it is written under its reduced form and 
is associated with a partial quotient $r'$ satisfying $r / p - 2 < r' < (r+2) / p$. 
If $p$ divides $a + b c_m$, then $( (a + b c_m) / p) / b$ is a convergent of $x_{m-1}$ 
written under its reduced form and 
associated with a partial quotient $r'$ satisfying $p r - 2 < r' < p (r+2)$.

Since $a$ and $b$ are coprime, $p$ cannot simultaneously divide $b$ and $a + b c_m$. 
Consequently, when $a/b$ is a convergent to $x_m$ associated with a partial quotient $r$ greater than $2p$, only 
three cases can occur, namely:
\smallskip 

$(i)$ $x_{m-1}$ has a convergent of denominator $b$ associated to a partial quotient 
in the interval $(r/p -2, r/p +1)$
and $x_{m+1}$ has a convergent of denominator $bp$ associated to a partial quotient 
in the interval $(r/p -2, r/p +1)$;

$(ii)$ $x_{m-1}$ has a convergent of denominator $b$ associated to a partial quotient 
in the interval $(r/p -2, r/p +1)$
and $x_{m+1}$ has a convergent of denominator $b$ associated to a partial quotient 
in the interval $(pr-2, p(r+2))$. 

$(iii)$ $x_{m-1}$ has a convergent of denominator $b/p$ associated to a partial quotient 
in the interval $(pr-2, p(r+2))$ 
and $x_{m+1}$ has a convergent of denominator $bp$ associated to a partial quotient 
in $(r/p -2, r/p +1)$. 

\smallskip


In Case $(i)$ we say that $r$ is a maximal partial quotient of $x_m$. 

The same argument shows that the following holds. 
Let $h$ be any positive integer such that $r \ge 2p^h$. 
If $p$ does not divide $b$, then 
$x_{m-h}$ has a convergent of denominator $b$ associated to a partial quotient $r_{m-h}$ 
with 
$$
r p^{-h} - 2 < r_{m-h} < (r+2) p^{-h}.
$$ 
If $p$ does not divide $a + b c_m$, then 
$x_{m+h}$ has a convergent of denominator $bp^h$ associated to a partial quotient $r_{m+h}$ with
$$
r p^{-h} - 2 < r_{m + h} < (r+2) p^{-h}. 
$$
Furthermore, if $x_m, x_{m+1}, \ldots , x_{m+h}$ have a convergent of denominator $b$, then $x_{m+h}$
has a partial quotient $r_{m+h}$ with 
$$
p^h r - 2 < r_{m+h} < p^h (r+2). 
$$

Set $r_m = r$ and, for $j > -m$, set $\teta_{m+j} = (\log r_{m+j}) / ( (m+j) \log p)$. 
Assuming that $r_m$ is a maximal partial quotient of $x_m$, we get
\begin{equation} \label{eqteta} 
\begin{aligned} 
( m- h) \teta_{m-h} = m \teta_m - h + o(1), \\
( m + h) \teta_{m + h} = m \teta_m - h + o(1), 
\end{aligned}
\end{equation}
for any integer $h$ with $0 \le h < m \teta_m$. 
Here and below, the notation $o(1)$ refers to a quantity that tends to $0$ as the indices tend to infinity. 
This shows that the function $n \mapsto \teta_n$ increases until $n=m$ and then decreases. 
In particular, if $\teta_m > 1/2$, then we get
$$
\teta_{\lfloor m/2 \rfloor} = 2 \teta_m - 1 + o(1), \quad 
\teta_{\lfloor 3m/2 \rfloor} = \frac{2 \teta_m - 1}{3} + o(1). 
$$

Let $m, n$ be integers with $m < n$. 
Inequalities \eqref{eqteta} imply that lower and upper bounds for the greatest partial 
quotient of $x_u$ where $m < u < n$ can be expressed  in terms of the greatest partial 
quotients of $x_m$ and $x_n$. A precise statement is as follows. 

\begin{proposition}  \label{etabound}
Keep the notation of Definition \ref{defeta}. 
Let $m, n$ be integers with $1 \le m < n$. 
For any integer $u$ with $m \le u \le  n$ we have
$$
\eta_u \le \frac{n(1 + \eta_n) - m (1 - \eta_m)}{n(1 + \eta_n) + m (1 - \eta_m)} + o(1). 
$$
Furthermore, if $x_m$ and $x_n$ have partial quotients $p^{\teta_m m}$ and $p^{\teta_n n}$,
respectively, then, for any integer $u$ with $m \le u \le n$ we have
$$
\eta_u \ge \frac{m(1 + \teta_m) - n (1 - \teta_n)}{m(1 + \teta_m) + n (1 - \teta_n)} + o(1). 
$$
\end{proposition}

\begin{proof}
We may assume that $p^{\eta_u u}$ is a maximal partial quotient of $x_u$. 
It then follows from \eqref{eqteta} that 
$$
m \eta_m = u \eta_u - (u - m) + o(1), \quad
n \eta_n = u \eta_u - (n - u) + o(1). 
$$
By eliminating $u$, this gives
$$
\frac{m(1 -  \eta_m) }{1 - \eta_u} 
= \frac{n( 1 + \eta_n) }{1 + \eta_u} + o(1)
$$
and
$$
\eta_u \le \frac{n(1 + \eta_n) - m (1 - \eta_m)}{n(1 + \eta_n) + m (1 - \eta_m)} + o(1). 
$$

Likewise, for $u = m+1, \ldots , n-1$, it follows from \eqref{eqteta} that 
$$
\eta_u \ge \max\Bigl\{  \frac{ m \teta_m - (u - m)}{u},  
 \frac{n \teta_n - (n-u)}{u}  \Bigr\} + o(1).
$$
This maximum attains its minimal value when both quantities are equal, that is, when 
$$
2 u = m (1 + \teta_m) + n (1 - \teta_n). 
$$
If this equation has no integer solution, we simply take the integer part and derive that 
$$
\eta_u \ge    \frac{2 m \teta_m +  m (1 - \teta_m) - n (1 - \teta_n) \bigr)}
{m (1 + \teta_m) + n (1 - \teta_n)} + o(1). 
$$
This implies the claimed lower bound. 
\end{proof}

The exponents $\tmu$ and $\eta$ are closely related to each other. 
As usual, $1 / 0$ means infinity. 

\begin{theorem} \label{equ}
Let $p$ be a prime number and $\bfc = (c_k)_{k \ge 0}$ a bounded 
sequence of integers. 
The $p$-adic number $\xi = c_0 + c_1 p + c_2 p^2 + \ldots$ satisfies
$$
\tmu (\xi) = {2 \over 1 - \eta (\bfc)} \quad \hbox{and} \quad
\htmu (\xi) = {2 \over 1 - \heta (\bfc)}.
$$
\end{theorem}

We point out that, in the statement of Theorem \ref{equ}, it is not assumed that $\bfc$ 
is the Hensel expansion of $\xi$. We only assume that $\bfc$ is bounded, but this is mostly by convenience. 

Intuitively, there is no reason for $\eta (\bfc)$ to be equal to $\eta(\bfc')$ when 
two sequences $\bfc$ and $\bfc'$ differ only by their first term, since $x_m$ 
may well have a very large partial quotient, while $x_m + \frac{1}{p^m}$ has none. 
Said differently, $\tmu (\xi)$ and $\tmu (\xi + 1)$ may well be different, as we have 
already noticed.

\begin{proof}
Without any restriction, we assume that $c_0$ is nonzero. 
For $m \ge 1$, set
$$
C_m = c_0 + c_1 p + c_2 p^2 + \ldots + c_{m-1} p^{m-1}. 
$$

Let $a, b$ be coprime integers satisfying
\begin{equation}\label{eq11}
|ab|\cdot \left|\xi - \frac{a}{b}\right|_p <\frac1{2}, \quad b \ge 1.  
\end{equation} 
Since $c_0$ is nonzero, $p$ does not divide the product $ab$. 
Let $m$ be the  
positive integer such that $|b\xi -a|_p = p^{-m}$.
Since $|\xi - C_m|_p \le p^{-m}$, we deduce that $p ^m$ divides $bC_m - a$.
Thus, there exists an integer $T$, which may be divisible by $p$, such that
$$
b C_m - a = p^m T.
$$
Furthermore, it follows from~\eqref{eq11} that $2|ab|<p^m$.
Consequently, we get
$$
\left| b {C_m \over p ^m} - T \right| = {|a| \over p^m} < {1 \over 2 b},
$$
and, by Legendre's theorem, 
$T/b$ is a convergent of
$C_m / p^m$. Note that $T/b$ is written in its lowest form, 
since $a$ and $b$ are coprime. Also, we have $b < p^m$. 

We can say a bit more. Write
$$
\frac{C_m}{p^m} = [0; r_1, \ldots , r_k], \quad \frac{T}{b} = [0; r_1, \ldots , r_j],
$$
with $j < k$ and $r_k = 1$ (recall that a rational number has two different continued fraction expansions, 
and only one of them terminates with $1$, except for $1 = [0;1]$). 
Then,
$$
\frac{1}{3 r_{j+1} b^2} \le {|a| \over b p^m} = \Bigl| \frac{C_m}{p^m}  - \frac{T}{b} \Bigr| \le \frac{1}{r_{j+1} b^2}, 
$$
giving that 
\begin{equation} \label{pq}
\frac{p^m}{3 b |a| } \le r_{j+1}  \le \frac{p^m}{ b |a| }.  
\end{equation} 
Define $\eta$ by $r_{j+1} = p^{\eta m}$. Then, we get
$$
|3ba|^{-1 / (1 - \eta)} \le |b \xi -a|_p = p^{-m} \le |ba|^{-1 / (1 - \eta)}. 
$$
Since, for every $\eps > 0$ there are integers $a, b, m$ as above with $m$ arbitrarily large and 
$p^{-m} < |a b|^{-\tmu (\xi) + \eps}$, this implies the inequality
$$
 \eta (\bfc) \ge 1 - {2 \over \tmu (\xi)}.  
$$

Let $m \ge 1$ be an integer. 
Let $T' /b' = [0; r_1, \ldots , r_h]$ denote any convergent to $C_m / p^m$ with $b' < p^m$. 
Set $a' = b' C_m - p^m T'$. Then, 
\begin{align*}
|b' \xi - a' |_p & = |b' \xi - b' C_m + p^m T' |_p  \\
& = |b' (c_m p^m + c_{m+1} p ^{m+1} + \ldots) + p^m T' |_p \le p^{-m}, 
\end{align*}
with equality if and only if $p$ does not divide $b' c_m + T'$. 
As above, we have 
$$
\frac{1}{3 r_{h+1} b'^2} \le {|a'| \over b' p^m} 
= \Bigl| \frac{C_m}{p^m}  - \frac{T'}{b'} \Bigr| \le \frac{1}{r_{h+1} b'^2}, 
$$
and, writing $r_{h+1} = p^{\eta' m}$, we get $|a' b'| \le  p^{ (1 - \eta') m}$ and 
$$
| b' \xi - a' |_p \le p^{-m} \le |a' b'|^{- 1 / (1 - \eta')}.
$$
This implies the inequalities
$$
\tmu (\xi) \ge {2 \over 1 - \eta (\bfc)}, \quad \htmu (\xi) \ge {2 \over 1 - \heta (\bfc)}. 
$$

The fourth inequality is slightly more delicate to establish.  
We introduce the sequence of multiplicative best approximation pairs 
$((a_{k}^{\times},b_{k}^{\times}))_{k \ge 1}$ to $\xi$. 

For a given $p$-adic number $\xi$ with 
\begin{equation} \label{notbad}
\inf_{a, b \not= 0} \, |ab| \cdot |b \xi - a|_p = 0 
\end{equation} 
(this can be assumed, since otherwise $\tmu (\xi) = \htmu (\xi) = 2$), 
we define the sequence of integer pairs $((\ta_{k}^{\astmult}, \tb_{k}^{\astmult}))_{k \ge 1}$ by 
taking a pair of {\em coprime} integers $(a,b)$ minimizing
$|b \xi- a|_{p}$ among all the 
integer pairs with $0<\sqrt{|ab|}\leq Q$, and letting the positive real number 
$Q$ grow to infinity. 
Write $\tQ_{k} = \sqrt{|\ta_{k}^{\astmult} \tb_{k}^{\astmult}|}$ for $k \ge 1$.
By construction, we have
$$
\tQ_{1} < \tQ_{2}  <\cdots, \quad |\tb_{1}^{\astmult}\xi - \ta_{1}^{\astmult}|_{p} 
> |\tb_{2}^{\astmult}\xi - \ta_{2}^{\astmult}|_{p} > \cdots.
$$
However, we cannot guarantee that 
$\tQ_{k}  |\tb_{k}^{\astmult}\xi - \ta_{k}^{\astmult}|_{p} 
> \tQ_{k+1}  |\tb_{k+1}^{\astmult}\xi - \ta_{k+1}^{\astmult}|_{p}$ for every $k \ge 1$. 
Therefore, we extract a subsequence $((\ta_{i_k}^{\astmult}, \tb_{i_k}^{\astmult}))_{k \ge 1}$ from 
$((\ta_{k}^{\astmult}, \tb_{k}^{\astmult}))_{k \ge 1}$, where $i_1 = 1$ and, for $k \ge 1$, the index 
$i_{k+1}$ is the smallest index $j > i_k$ such that $\tQ_j  |\tb_{j}^{\astmult}\xi - \ta_{j}^{\astmult}|_{p} 
< \tQ_{i_k}  |\tb_{i_k}^{\astmult}\xi - \ta_{i_k}^{\astmult}|_{p}$. 
This gives an infinite subsequence since $\xi$ satisfies \eqref{notbad}. 

To simplify the notation, put $a_k^{\astmult} = \ta_{i_k}^{\astmult}$, $b_k^{\astmult} = \tb_{i_k}^{\astmult}$, 
and $Q_k = \tQ_{i_k}$, for $k \ge 1$. 

Observe that
$$
\mu^{\times}(\xi) = \limsup_{k \to \infty} 
\, \frac{- \log |b_{k}^{\times}\xi - a_{k}^{\times}|_{p}}{\log Q_k} 
$$
and
\begin{equation} \label{muhat}
\widehat{\mu}^{\times}(\xi) 
=  \liminf_{k \to \infty} \, \frac{- \log |b_{k}^{\times}\xi - a_{k}^{\times}|_{p} + 
\log (Q_{k+1} / Q_{k})}{\log Q_{k+1}}. 
\end{equation} 

Take two consecutive best approximation pairs $(a_{k}^{\times},b_{k}^{\times})$, 
$(a_{k+1}^{\times},b_{k+1}^{\times})$, and define $m_k$ and $m_{k+1}$ by
$$
|b_{k}^{\times}\xi - a_{k}^{\times}|_{p} = p^{-m_k}, \quad 
|b_{k+1}^{\times}\xi - a_{k+1}^{\times}|_{p} = p^{-m_{k+1}}. 
$$
It follows from \eqref{pq} that the largest partial quotients of $x_{m_k}$ and 
$x_{m_{k+1}}$ are $\asymp p^{m_{k}}  Q_{k}^{-2}$ and 
$\asymp p^{m_{k+1}}  Q_{k+1}^{-2}$, respectively. 
We apply the second inequality of Proposition \ref{etabound} with 
$$
m=m_k, \quad \teta_m = 1 - \frac{\log Q_k^2}{m_k \log p}, \quad 
n=m_{k+1}, \quad  \teta_n = 1 - \frac{\log Q_{k+1}^2}{m_{k+1} \log p}
$$
to derive that any $x_u$ with $m_k < u < m_{k+1}$ has a partial quotient 
at least as large as $p^{\eta_u u}$ with
\begin{align*}
\eta_u & \ge \frac{(2 m_k \log p - 2 \log Q_k)- 2 \log Q_{k+1}}{(2 m_k \log p - 2 \log Q_k) + 2 \log Q_{k+1}} 
+ o(1) \\
& = 1 - \frac{2 \log Q_{k+1}}{m_k \log p + \log Q_{k+1}/Q_k} + o(1). 
\end{align*}
Recalling that $m_k \log p = - \log |b_{k}^{\times}\xi - a_{k}^{\times}|_{p}$, it then follows from \eqref{muhat} 
that 
$$
\heta(\bfc) \ge  1 - \frac{2}{\htmu (\xi)}. 
$$
This completes the proof of the proposition. 
\end{proof}

In the course of the proof of Theorem \ref{equ}, we have obtained the following statement.

\begin{proposition}   \label{precis}
If there exist positive real numbers $c_1, \delta$ such that, for every $m \ge 1$, all the 
partial quotients of $x_m$ are less than $c_1 p^{\delta m}$, then there exists $c_2 > 0$ such that 
$$
|b \xi - a |_p > c_2 |ab|^{-1 / (1 - \delta)}, \quad \hbox{for all $a, b$}. 
$$
If there exist positive real numbers $c_3, \delta$ and arbitrarily large $m$ such that $x_m$ has a 
partial quotient greater than $c_3 p^{\delta m}$, then there exist $c_4 > 0$ and integers $a, b$
with $|ab|$ arbitrarily large, such that 
$$
|b \xi - a |_p < c_4 |ab|^{-1 / (1 - \delta)}. 
$$
\end{proposition}

\section{Very good rational approximations to the $p$-adic Thue--Morse number} \label{sec:4} 

In this section, we use combinatorial properties of the Thue--Morse word $\bft$ 
to establish \eqref{good} and to exhibit an infinite family of rational numbers $z_m$ (see \eqref{defz} 
for their definition) 
having a very large partial quotient. 

\begin{proof}[Proof of \eqref{good}]
Observe that 
$$
(1 + p^2)  \xitm  = 1  - p - 2 p^4 + 2 p^5 + 2 p^{12} - 2 p^{13} + \ldots ,
$$
where the coefficients of $p^6$ up to $p^{11}$ are $0$. 
More generally, for $k \ge 1$, we get
$$
(1 + p^{2^k}) \xitm   = R_k (p) + 2 (-1)^{k+1} p^{3 \cdot 2^{k+1}} +  p^{3 \cdot 2^{k+1} + 1} s_k,
$$
where $s_k$ is a nonzero element of $\Z_p$ 
and $R_k (X)$ is a polynomial with coefficients in $\{0, \pm 1, \pm 2\}$
and of degree $3 \cdot 2^k - 1$. 

This can be checked either by using the substitution $\tau$, or by a direct computation based on the 
recursion defining $\bft$. 
Namely, we observe that 
$$
t_j + t_{j+2} = 0, \quad \hbox{for $j = 4, 5, \ldots , 9$, that is, for $j=2^2, \ldots , 2^3 + 2 - 1$,} 
$$
and
$$
t_3 + t_5 = t_{10} + t_{12} = 2, \quad t_{11} + t_{13} = -2. 
$$
Furthermore, $t_j + t_{j+2} = 0$ implies that $t_{2j}+ t_{2(j+2)} = t_{2j} + t_{2j + 4} = 0$ and 
$$
 t_{2j+1} + t_{(2j + 1) + 4} =  - t_j  - t_{j+2} = 0.  
$$
Consequently, we derive that 
$$
t_j + t_{j+2^k} = 0, \quad \hbox{for $k \ge 1$ and $j=2^{k+1}, \ldots , 2^{k+2} + 2^k - 1$.}
$$
In addition, we check that 
$$
t_{2^{k+1} - 1} +  t_{2^{k+1} + 2^k - 1} = 2 \cdot (-1)^{k+1}, 
\quad \hbox{for $k \ge 1$,} 
$$
and
$$
t_{2^{k+2} + 2^k} +  t_{2^{k+2} + 2^{k+1}} = 2, \quad 
t_{2^{k+2} + 2^k + 1} +  t_{2^{k+2} + 2^{k+1} + 1} = -2, 
\quad \hbox{for $k \ge 1$}.
$$
We get eventually
$$
|  (1 + p^{2^k}) \xitm - R_k (p) |_p = p^{- 3 \cdot 2^{k+1}}, \quad \hbox{for $p \ge 3$}, 
$$
while
$$
|  (1 + 2^{2^k}) \xi_{\bft, 2} - R_k (2) |_2 = 2^{- 3 \cdot 2^{k+1} -1}. 
$$
For $k \ge 1$, putting
$$
b_{k, p} =  1 + p^{2^k}, \quad a_{k, p} = R_k (p),
$$
we check that 
$$
|b_{k, p}| \le p^{2^k + 1}, \quad |a_{k, p}| \le 2 p^{3 \cdot 2^{k} + 1}, 
$$
and  
$$
|b_{k, p} \xitm - a_{k, p} |_p \le  p^{- 3 \cdot 2^{k+1}} 
\le \bigl( 2 p^2 |a_{k, p} b_{k, p}|^{-1} \bigr)^{3/2} \le 4 p^3 |a_{k, p} b_{k, p}|^{-3/2}. 
$$
This establishes \eqref{good} and implies that $\tmu (\xitm) \ge 3$. 
\end{proof}

It follows from Proposition \ref{precis} that every rational number $z_{3 \cdot 2^k}$ 
has a large partial quotient. 

\begin{proposition} \label{largeqp}
There exist a positive real number $c$ 
such that, for every $k \ge 1$, one among the rational numbers 
$z_{3 \cdot 2^{k}}$ and $z_{3 \cdot 2^{k} + 1}$ has a maximal partial quotient in $[c^{-1} p^{2^k}, c p^{2^k} ]$
associated with a convergent whose denominator is $2^{2^{k-1}} + 1$ if $p=2$ and 
$(p^{2^{k-1}} + 1)/2$ if $p$ is odd. 
\end{proposition}

\begin{proof}
Observe that 
$$
\biggl( 1 + {1 \over p^2} \biggr) z_{12} = - {1 \over p} + \frac{1}{p^2} +  {2 \over p^9} - {2 \over p^{10}}
- {1 \over p^{13}} + {1 \over p^{14}}.
$$
Since $t_0 = t_5$, for any $k \ge 1$, the prefix of $\bft$ of length $2^k$ is equal to the suffix of length $2^k$
of the prefix of $\bft$ of length $6 \cdot 2^k$.  Consequently, we have 
$$
\biggl( 1 + {1 \over p^{2^k}} \biggr) z_{3 \cdot 2^{k+1}} = {T_k (p) \over p^{2^k}} 
+ \frac{2}{p^{2^{k+2} + 1}} -  \frac{2}{p^{2^{k+2} + 2}} + \ldots , 
$$
where $T_k (X) = t_0 + t_1 X + \ldots + t_{2^k - 1} X^{2^k - 1} =   \prod_{j=0}^{k-1} ( 1 - X^{2^j} )$.  
This implies
\begin{equation} \label{grandqp} 
{1 \over p^{3 \cdot 2^{k}}} \ll \bigl| (p^{2^k} + 1) z_{3 \cdot 2^{k+1}} - T_k (p) \bigr| 
\le {2 \over p^{3 \cdot 2^{k}}}.
\end{equation}
Thus, there must be a very large partial quotient in the continued fraction expansion of $z_{3 \cdot 2^{k+1}}$. 
For $h = 0, \ldots , k-1$, since $p^{2^h} - 1$ divides $p^{2^k} - 1$, we see that 
$$
\gcd (p^{2^h} - 1, p^{2^k} + 1) \ \ \hbox{divides $2$}.  
$$
Furthermore, $4$ does not divide $p^{2^k} + 1$. We conclude that 
$T_k (p)$ and $p^{2^k} + 1$ are coprime for $p=2$, while 
their greatest common divisor is $2$ for $p \ge 3$. 
This shows that $z_{3 \cdot 2^{k+1}}$ has a partial quotient $r_{3 \cdot 2^{k+1}}$ 
with $r_{3 \cdot 2^{k+1}}   \asymp p^{2^k}$. 

Note that $t_{3 \cdot 2^{k+1}} = t_3 = 1$. 
If $p \ge 3$, then $p$ does not divide $T_k (p) + (p^{2^k} + 1)$ and we conclude 
that $r_{3 \cdot 2^{k+1}}$ is a maximal partial quotient.  
For $p=2$ we check that 
$z_{3 \cdot 2^{k+1} + 1}$ has a maximal partial quotient. 
This concludes the proof. 
\end{proof}

\medskip

A deeper study of the combinatorial properties of $\bft$ shows that, 
for $j \ge 0$ and $k \ge 1$, there are polynomials $R_{k,j} (X)$ of degree at most equal to 
$2^{k-1} (6 + j 2^4) - 1$ such that 
$$
p^{- 2^{k-1} (12 + j 2^4)} \ll |  (1 + p^{2^k}) \xitm - R_{k,j} (p)  |_p \le p^{- 2^{k-1} (12 + j 2^4)}. 
$$
This shows that, for every $j \ge 0$, there are integers $a, b$ with $|ab|$ arbitrarily large such that 
$$
|b \xitm - a|_p  \asymp  |ab|^{- (3  + 4 j) /  (2  + 4 j) }.
$$
The exponents form the sequence of rational numbers $3/2, 7/6, 11/10,  \ldots$ 
We suspect that, for any given $\eps > 0$, all but finitely many solutions to
$$
|b \xitm - a|_p  <  |ab|^{-1 - \eps}
$$
belong to the families described above.

\section{Use of Hankel determinants}   \label{sec:5}

For $k \ge 1$, let 
$$
z_{3 \cdot 2^k}  = {t_{3 \cdot 2^k-1} \over p} + {t_{3 \cdot 2^k-2} \over p^2} + \ldots 
+ {t_0 \over p^{3 \cdot 2^k}} = [0; d_{1, k}, \ldots , d_{\ell(k), k}]
$$
denote the continued fraction expansion of $z_{3 \cdot 2^k}$ with $d_{\ell(k), k} \ge 2$. 
By Proposition \ref{largeqp}, there exists $m(k)$ such that 
\begin{equation}  \label{Largeqp}
d_{m(k), k} \asymp p^{ 2^{k}} \quad \hbox{and} \quad 
{\rm{denominator}} ([0; d_{1, k}, \ldots , d_{m(k) - 1, k}] ) \asymp p^{2^{k-1}}. 
\end{equation}
To establish the upper bound $\tmu (\xitm) \le 3$, we need to ensure that
the second largest partial quotient of $z_{3 \cdot 2^k}$ is much smaller than $d_{m(k), k}$. 
To do this, we use Hankel determinants in a similar spirit as in 
\cite{Bu11,BuHan21}.

The purpose of this section is to establish the following statement. 

\begin{proposition}  \label{boundqp}
There exists a positive constant $c$ such that, 
for every sufficiently large $k$, every partial quotient of $z_{3 \cdot 2^k}$ 
different from $d_{m(k), k}$ 
is at most equal to $c p^{2^{k-1}}$.
\end{proposition}

By \eqref{Largeqp}, the partial quotients $d_{1, k}, \ldots , d_{m(k) - 1, k}$ are all $\ll p^{2^{k-1}}$.
To establish Proposition \ref{boundqp}, it thus remains for us to 
bound from above the partial quotients $d_{m(k) + 1, k}, \ldots , d_{\ell(k), k}$. 
To this end, by \eqref{Largeqp}  it is sufficient to consider only the 
convergents $a / b$ of $z_{3 \cdot 2^k}$ with $b \gg p^{3 \cdot 2^{k-1}}$ and to show that 
none of them is associated with a partial quotient $\gg p^{ 2^{k-1} }$.

The method of the proof and some additional computation yields a stronger conclusion, with 
$p^{ 2^{k - 1} }$ 
replaced by $p^{  2^{k - h} }$, for some integer $h \ge 3$. 
It is even likely that the following statement holds:

\smallskip

{\it For every $\eps > 0$ and every sufficiently large $k$, 
every partial quotient of $z_{3 \cdot 2^k}$ 
different from $d_{m(k), k}$ is at most equal to $p^{\eps 2^k}$.}

\smallskip


We follow very closely the argumentation of \cite{Bu11}, where Pad\'e 
approximants are used to construct a dense, in a suitable sense, sequence of good rational approximations 
to the real Thue--Morse--Mahler numbers. 
However, our problem is different, since we have to control the partial quotients 
of the {\it rational numbers} $z_{3 \cdot 2^k}$. 

As in \cite{Bu11,BuHan21} we work with the Thue--Morse sequence written over 
the alphabet $\{-1, 1\}$. This is at this step of the proof that the choice of the alphabet does matter.

\begin{proof}[Proof of Proposition \ref{boundqp}]
As in \cite{Bu11}, we briefly recall several basic facts on
Pad\'e approximants. 
We refer the reader
to \cite{Br80,BaGr96} for the proofs and for 
additional results.
Let 
$$
f(z) = \sum_{k \ge 0}  \, c_k z^k, \quad c_k \in \Q,
$$
be a power series. 
Let $u, v$ be non-negative integers.
The Pad\'e approximant $[u / v]_f (z)$ is any rational
fraction $A(z)/B(z)$ in $\Q[[z]]$ such that
$$
\deg(A) \le u, \quad \deg(B) \le v, 
\quad \hbox{and $\ord_{z=0} (B(z) f(z) - A(z)) \ge u + v + 1$}.
$$
For $k \ge 1$, let
$$
H_k (f) := \left|
\begin{matrix}
c_0 & c_1 & \ldots & c_{k-1} \\
c_1 & c_2 & \ldots & c_k \\
\ \vdots \hfill & \ \vdots \hfill & \ddots &
\ \vdots \hfill \\
c_{k-1} & c_k & \ldots & c_{2k-2} \\
\end{matrix}
\right|
$$
denote the Hankel determinant of order $k$
associated to $f(z)$.
If $H_k (f)$ is non-zero, then the Pad\'e approximant
$[k-1 / k]_f (z)$ exists and we have
$$
f(z) - [k-1 / k]_f (z) = {H_{k+1} (f) \over H_k (f)} \, z^{2k}
+ O(z^{2k+1}).   
$$

For a positive integer $k$, set
$$
\tig_0 (z) = 1 + z - z^2 = - (t_2 + t_1 z + t_0 z^2) 
$$
and
$$
\tig_k (z) = (1 - z) (1 - z^2) \cdots ( 1 - z^{2^{k-1}}) \tig (z^{2^k}), \quad k \ge 1. 
$$
The definition of $\bft$ implies that 
$$
\tig_k (z) = (-1)^{k+1} (t_{3 \cdot 2^{k} - 1} + t_{3 \cdot 2^{k} - 2} z + \ldots + t_0 z^{3 \cdot 2^{k} - 1}), 
\quad k \ge 0,  
$$
thus
$$
z_{3 \cdot 2^k} = (-1)^{k+1}  \frac{\tig_k(1/p)}{p},   \quad k \ge 0. 
$$
It is sufficient for our purpose to show that the continued fraction expansion of the rational 
number $\tig_k (1/p)$ has no `too large' partial quotient associated 
with a convergent of denominator $\gg p^{3 \cdot 2^{k-1}}$. 
Note that the fact that $p$ is prime 
does not play any r\^ole in this section
and all what follows also holds for the rational number $\tig_k (1/b)$, where 
$b \ge 2$ is an integer. 

Let $K \ge 2$ be an integer to be fixed later. 
Assume that we have checked that
$$
H_j (\tig_K) \not= 0, \quad j= 3 \cdot 2^{K-1} + 1, \ldots, 3 \cdot 2^{K}.
$$
Consequently, there exist integer polynomials $P_{j,0} (z)$,
$Q_{j,0} (z)$ of degree at most $j - 1$ and $j$, respectively,
and a non-zero rational number $h_j$ such that
$$
{\tig_K}(z) - P_{j,0}(z)/Q_{j,0}(z) = h_j z^{ 2 j} + O(z^{ 2 j + 1}), \quad 3 \cdot 2^{K-1} + 1 \le j \le 3 \cdot 2^{K} - 1. 
$$
The real numbers $c_1, c_2, \ldots $ occurring below are all positive and depend only at most on $K$
(note that the index $j$ is bounded from above and from below in terms of $K$). 
There exists $c_1$ such that
\begin{equation} \label{4.2}
\biggl| {\tig_K}(z^{2^m}) - {P_{j,0}(z^{2^m}) \over Q_{j,0}(z^{2^m}) }
- h_j z^{2^{m+1} j} \biggr|
\le c_1 z^{2^{m+1} j + 2^m},   
\end{equation}
for $0 < z \le 1/2$ and $3 \cdot 2^{K-1} + 1 \le j \le 3 \cdot 2^{K} - 1$. 
An immediate induction yields
$$
(1 - z) (1 - z^2) \cdots (1 - z^{2^{m-1}}) \tig_K (z^{2^m}) = \tig_{K+m} (z).
$$
Set
$$
P_{j,m} (z) = \prod_{h=0}^{m-1} (1 - z^{2^h}) \, 
P_{j,0}(z^{2^m}), \quad 
Q_{j,m} (z) = Q_{j,0}(z^{2^m}).
$$
Note that $P_{j,m}(z) / Q_{j,m}(z)$ is the Pad\'e 
approximant $[2^m j - 1 / 2^m j]_{\tig_{K+m}} (z)$.

By multiplying both members of \eqref{4.2} by $(1 - z) (1 - z^2) \ldots (1 - z^{2^{m-1}})$, we obtain
$$
\biggl| \tig_{K+m} (z) - {P_{j,m}(z) \over Q_{j,m}(z)}
- h_j \prod_{h=0}^{m-1} (1 - z^{2^h}) \, z^{2^{m+1} j} \biggr|   \le c_2 z^{2^{m+1}j+2^m},   
$$
for $0 < z \le 1/2$ and $3 \cdot 2^{K-1} + 1 \le j \le 3 \cdot 2^{K} - 1$.

Evaluating at $z = 1/p$ and 
arguing as in \cite{Bu11}, there exist an absolute, positive $\xig$ and an integer $m_0$, depending 
only on $K$, such that 
the inequalities 
\begin{equation} \label{4.7}
{h_j \xig  \over 2} \, p^{- 2^{m+1} j} \le 
\biggl| \tig_{K+m} (1/p) - {P_{j,m}(1/p) \over Q_{j,m}(1/p)} 
\biggr| \le {3 h_j \over 2} \, p^{- 2^{m+1} j}.  
\end{equation}
hold for $m > m_0$ and $3 \cdot 2^{K-1} + 1 \le j \le 3 \cdot 2^{K} - 1$. 

Define the integers
$$
p_{j,m} = p^{2^m j} P_{j,m}(1/p), \quad 
q_{j,m} = p^{2^m j} Q_{j,m}(1/p).
$$
There exist 
$c_3, \ldots , c_8$ 
such that
\begin{equation} \label{4.8}
c_3  p^{2^m j} \le q_{j, m} \le c_4 p^{2^m j}, 
\end{equation}
\begin{equation} \label{4.9}
{c_5 \over p^{2^{m+1}j}} \le
\biggl| \tig_{K+m} (1/p) - {p_{j, m} \over q_{j, m}} \biggr| \le 
{c_6 \over p^{2^{m+1}j}},  
\end{equation}
and, by combining \eqref{4.8} and \eqref{4.9},
\begin{equation} \label{4.10}
{c_7 \over q_{j, m}^2} \le
\biggl| \tig_{K+m} (1/p)  - {p_{j, m} \over q_{j, m}} \biggr| \le 
{c_8 \over q_{j, m}^2},   
\end{equation}
for $m > m_0$ and $3 \cdot 2^{K-1} + 1 \le j \le 3 \cdot 2^{K} - 1$. 


Let $r/s$ be a convergent to $\tig_{K+m} (1/p)$ with $s > c_9 p^{3 \cdot 2^{K+m - 1}}$, for 
some absolute positive real number $c_9$.  
Assume that there is $j$ with $3 \cdot 2^{K-1} + 1 \le j \le 3 \cdot 2^K - 2$ such that 
$$
q_{j, m} \le 2 c_8 s < q_{j+1, m}.
$$
Then,
\begin{align*}
\Bigl| \tig_{K+m} (1/p) - {r \over s} \Bigr| & \ge \Bigl|  {r \over s} -  {\tip_{j+1,m} \over q_{j+1,m}}  \Bigr| 
- \Bigl|  \tig_{K+m} (1/p) -  {\tip_{j+1,m} \over q_{j+1,m}}  \Bigr|  \\
& \ge  {1 \over s q_{j+1,m}} - {c_8 \over q_{j+1,m}^2} \ge {1 \over 2 s q_{j+1,m}}  
\ge {1 \over c_{10} s^{2 + 1/j}},
\end{align*}
since $q_{j+1,m} \le c_{11} q_{j,m}^{1 + 1/j} \le c_{12} s^{1 + 1/j}$. 

Since $s > c_9 p^{3 \cdot 2^{K+m - 1}}$, the case $2 c_8 s \le q_{3 \cdot 2^{K-1} + 1,m}$ 
can be treated analogously. Furthermore, 
if $2 c_8 s \ge q_{3 \cdot 2^K -1, m}$, then the partial 
quotient $A$ associated with the convergent $r/s$ satisfies
$$
\frac{1}{s p^{3 \cdot 2^{K+m} - 1}} \le \Bigl|  \tig_{K+m} (1/p) - \frac{r}{s} \Bigr | \le \frac{1}{A s^2},
$$
thus
$$
A \le \frac{p^{3 \cdot 2^{K+m} - 1}}{s} \le c_{13} p^{2^{m}} \le c_{14} s^{1 / (3 \cdot 2^K -1)}.
$$
This shows that, for $m$ large enough, the second largest partial quotient of the rational number 
$z_{3 \cdot 2^{m + K}} = (-1)^{k+1} \tig_{K+m} (1/p) / p$ is at most equal to 
$p^{2^{m +1} \eta}$, for $\eta = 1/ (3 \cdot 2^{K-1})$.

An easy calculation shows that, for
$$
\tig_2 (z) = 1 -z -z^2 + z^3 + z^4 - z^5 -z^6 + z^7 - z^8 + z^9 + z^{10} - z^{11},
$$
we have $H_2(\tig_2) = -2$, $H_3(\tig_2) = \ldots = H_6 (\tig_2) = 0$, 
$H_7 (\tig_2) = 64$, $H_8 (\tig_2) = 128$, 
$H_9 (\tig_2) = - 64$, $H_{10} (\tig_2) = -56$, $H_{11} (\tig_2) = -14$, $H_{12} (\tig_2) = 1$. 
Consequently, we can take $K=2$ in the above computation and we get $\eta= 1/6$, as
announced.  

A rapid check shows that the Hankel determinants $H_{25} (\tig_4), \ldots , H_{48} (\tig_4)$ do not 
vanish, thus we can take $K=4$ and conclude that, for $k$ large enough, 
every partial quotient $d_{j,k}$ of $z_{3 \cdot 2^k}$ with $j > m(k)$ 
is at most equal to $c p^{3 \cdot 2^{k} / 24}$ for some positive constant $c$. 
\end{proof}

\section{Proof of Theorem \ref{th:main}}  \label{sec:6} 

Recall that $z_m$ is defined in \eqref{defz}. The following statement, which partly gathers 
results from Sections \ref{sec:4} and \ref{sec:5}, is the 
key ingredient for the proof of Theorem \ref{th:main}. 

\begin{proposition} \label{propmain}
For every positive real number $\eps$ and for every sufficiently large integer $k$, 
all the partial quotients of $z_{2^k}$ are less than $p^{\eps 2^k}$. 
There exists a positive real number $c$ such that, for every positive integer $k$, 
the rational number $z_{3 \cdot 2^k}$ has a partial quotient in the 
interval $[c^{-1} p^{2^k}, c p^{2^k}]$, while all its other partial 
quotients are less than $c p^{  2^{k - 1} }$.
\end{proposition}

The first statement of Proposition \ref{propmain} is a direct 
consequence of \cite[Th. 1.1]{BuHan21}, reproduced below.

\begin{theorem} \label{th:BuHan} 
There exists a positive real number $K$ such that, 
for every integer $b \ge 2$ and every integer $\ell \ge 2$, the inequality
$$
\Bigl| \prod_{h = 0}^\ell \, (1 - b^{-2^h}) - {p \over q} \Bigr| 
> {1 \over q^2 \exp( K \log b \, \sqrt {\log q \, \log \log q})},
$$
holds for every rational number $p/q$ different from $\prod_{h = 0}^\ell \, (1 - b^{-2^h})$. 
\end{theorem}

\begin{proof}[Proof of 
Proposition \ref{propmain}] 
Observe that the prefix of length $4$ of $\bft$ is a palindrome and that 
$\tau^2 (1) = 1 -1 - 1 1$ and $\tau^2 (-1) = -1 1 1 -1$ are both palindromes. 
Consequently, for $k \ge 1$, the prefix of length $4^k$ of $\bft$ is a palindrome and
\begin{align*}
z_{4^k} &= {t_{4^k-1} \over p} + {t_{4^k-2} \over p^2} + \ldots + {t_0 \over p^{4^k}} \\
& = {t_{0} \over p} + {t_{1} \over p^2} + \ldots + {t_{4^k - 1} \over p^{4^k}}
= p^{-1} \, \prod_{h = 0}^{2k-1} \, (1 - p^{-2^h}).
\end{align*}

Define the involution $\sigma$ by $\sigma(1) = -1$ and $\sigma(-1) = 1$. 
It follows from the definition of $\bft$ that its prefix of length $2 \cdot 4^k$  
is equal to the concatenation of its prefix $T_{4^k}$ of length $4^k$ with 
$\sigma(T_{4^k})$. Consequently,  
$$
t_{2 \cdot 4^k - 1} \ldots t_1 t_0 = \sigma(t_0 t_1 \ldots t_{2 \cdot 4^k - 1}) 
$$
is the prefix of length $2 \cdot 4^k$ of the Thue--Morse word obtained from $\bft$ 
after exchanging $1$'s and $-1$'s and we get 
$$
z_{2 \cdot 4^k} = - p^{-1} \, \prod_{h = 0}^{2k} \, (1 - p^{-2^h}), \quad k \ge 0. 
$$
The first assertion of  Proposition \ref{propmain} then follows from Theorem \ref{th:BuHan}.

The second assertion has been established in Proposition \ref{largeqp} 
and the last one in Proposition \ref{boundqp}. 
\end{proof}

\begin{proof}[Completion of the proof of Theorem \ref{th:main}] 
The first assertion of  Proposition \ref{propmain} implies that $\heta (\bft) = 0$, which, by 
Theorem \ref{equ}, gives that $\htmu (\xitm) = 2$. 

Since we have already established \eqref{good}, it only remains for us to 
prove \eqref{notgood}, that is, by the first assertion of 
Proposition \ref{precis}, to show that there exists a positive real number $c$
such that, for every $m \ge 1$, all the partial quotients of $z_m$ are less than $c p^{m/3}$.

By Proposition \ref{largeqp}, we know that, for $k \ge 1$, there is a maximal partial quotient 
of size $\asymp p^{2^k}$ attained at the rational number $z_{3 \cdot 2^k}$
or $z_{3 \cdot 2^k + 1}$. It remains for us to control the partial 
quotients which do not derive from this maximal partial quotient. 
Therefore, it is sufficient to bound from above any other maximal partial quotient attained at some $z_u$ 
and written under the form $p^{\eta_u u}$.  

Let $\eps$ be a small positive real number. 
For $k$ large enough, all the partial quotients of $z_{2^k}$ are $< p^{\eps 2^{k}}$
and all the partial quotients of $z_{3 \cdot 2^k}$ except the largest 
one are $< p^{  2^{k-1}  + \eps }$. 
By Proposition \ref{etabound} applied with $m = 2^k$ and $n = 3 \cdot 2^{k - 1}$, we get 
that every $\eta_u$ with $2^k \le u  \le  3 \cdot 2^{k - 1}$ satisfies 
$$
\eta_u \le \frac{3 (1 + 1/6 + \eps) - 2 (1 - \eps)}{3 (1 + 1/6 + \eps) + 2 (1 - \eps)} + o(1) 
\le \frac{3}{11} + 5 \eps + o(1),
$$
where, as below, $o(1)$ denotes some quantity which tends to $0$ as $k$ tends to infinity. 
Similarly, every $\eta_u$ with $3 \cdot 2^{k - 1} \le u \le 2^{k + 1}$ satisfies 
$$
\eta_u \le \frac{ 4 (1 + \eps) - 3 (1 - 1/6 - \eps) }{3 (1 - 1/6 - \eps) + 4 (1 + \eps)} + o(1) 
\le \frac{3}{13} + 7 \eps + o(1).
$$
Taking $\eps$ small enough and $k$ large enough, both upper bounds are less than $3/10$. 
Consequently, we have proved that, for $m$ sufficiently large, all the partial quotients of $z_m$ are 
bounded from above by some positive constant times $p^{m/3}$. 
We then apply Proposition \ref{precis} to get \eqref{notgood} and Theorem \ref{equ} 
to derive that $\tmu (\xitm) \le 3$. 
\end{proof}

As indicated above, the fact that $\htmu (\xitm) = 2$ is a direct consequence of \cite{BuHan21}, 
which is also used in the proof that $\tmu (\xitm) \le 3$. However, the latter proof can be made independent 
of \cite{BuHan21}. Namely, by arguing as in Section \ref{sec:5} and checking that 
some Hankel determinants are nonzero, we can prove that, for $k$ large 
enough, every partial quotient of $z_{2^k}$ is at most equal to $c p^{2^k / 24}$ for some 
positive constant $c$. Then, we use Proposition \ref{etabound} in a similar way as above.

\section{Rational approximation to the $p$-adic Fibonacci number}   \label{sec:7}

The Fibonacci word 
$\bff = f_1 f_2 f_3 \ldots$ over $\{0, 1\}$ is the limit of the sequence of finite words 
$$
0, 01, 010, 01001, \ldots , 
$$
starting with $0, 01$, and such that, for $n \ge 1$, its $(n+2)$-th element is the concatenation of its 
$(n+1)$-th and its $n$-th elements. 
We then have 
$$
\bff = 010010100100101001010 \ldots
$$
Said differently, $\bff$ is the fixed point of the substitution $\varphi$ defined by $\varphi (0) = 01$
and $\varphi(1) = 0$. Note that we have
\begin{equation} \label{sigma}
\varphi(f_1 \ldots f_{F_n}) = f_1 \ldots f_{F_{n+1}}, \quad n \ge 1,
\end{equation} 
where $(F_n)_{n \ge 0}$ denote the Fibonacci sequence defined by $F_0 = 0$, $F_1 = 1$ 
and the recursion $F_{n+2} = F_{n+1} + F_n$ for $n \ge 0$. 
Let 
$$
\xifi = \sum_{n \ge 1} \, \frac{f_n}{p^n} = \frac{1}{p^2} + \frac{1}{p^5} + \frac{1}{p^7} + \frac{1}{p^{10}} + \ldots
$$ 
denote the $p$-adic number whose Hensel expansion is given by the 
Fibonacci word $\bff$ over $\{0, 1\}$. 

The existence of long repetitions at and near the beginning of $\bff$ allows us to 
establish the following result. 
Let $\gamma = (1 + \sqrt{5})/2$ denote the Golden Ratio. 

\begin{theorem} \label{fibobound}
The $p$-adic Fibonacci number $\xifi$ satisfies
$$
\mu(\xifi) = \gamma^2, \quad 
\tmu(\xifi) \ge \gamma^2 + \frac{1}{1 + \gamma^2}, \quad 
\hbox{and} \quad \htmu(\xifi) \ge 2 + \frac{1}{2 \gamma}.
$$
\end{theorem}

To establish the lower bounds in Theorem \ref{fibobound}, 
we use the fact that there are many repetitions in $\bff$, which yield very good rational 
approximations to $\xifi$. Our key 
auxiliary result is the following proposition. 

\begin{proposition} \label{fiborep} 
For $n \ge 4$, the word
$$
f_1 \ldots f_{F_n} f_1 \ldots f_{F_n} f_1 \ldots f_{F_{n-1} - 2}
$$
of length $F_{n+2} - 2$ is the longest common prefix of the word $\bff$ and 
the word $(f_1 \ldots f_{F_n})^\infty$.
For $n \ge 4$, the word
$$
f_1 \ldots f_{F_{n+1}} f_1 \ldots f_{F_n} f_1 \ldots f_{F_n} f_1 \ldots f_{F_n} f_1 \ldots f_{F_{n-1} - 2}
$$
of length $2 F_{n+2} - 2$ 
is the longest common prefix of the word $\bff$ and the word 
$f_1 \ldots f_{F_{n+1}} (f_1 \ldots f_{F_n})^\infty$. 
\end{proposition}

\begin{proof}
This follows from \eqref{sigma} by induction, after having checked, also by induction, that 
$f_{F_n -2} f_{F_n-1} f_{F_n} = 001$ if $n \ge 4$ is odd 
and $f_{F_n - 2} f_{F_n - 1} f_{F_n} = 010$ if $n \ge 4$ is even. 
We omit the details. 
\end{proof}

\begin{proof}[Proof of Theorem \ref{fibobound}]
It follows from Proposition \ref{fiborep} that,
for $n \ge 1$, there exist integers $r_n$ and $x_n$ such that 
$$
| (p^{F_n} - 1) \xifi - r_n|_p \asymp_p p^{-F_{n+2}}, \quad |r_n| \asymp_p p^{F_n},
$$
and 
$$
| (p^{F_n} - 1) \xifi - x_n|_p \asymp_p p^{- 2 F_{n+2}}, \quad |x_n| \asymp_p p^{F_{n+2}}. 
$$
Setting $y_n = p^{F_{n}} - 1$ for $n \ge 1$, we check that 
$$
|r_n y_n|  \asymp_p p^{2 F_{n}}, \quad 
|x_n y_n| \asymp_p p^{F_{n + 2} + F_n}, \quad n \ge 1. 
$$
Since $F_{n+1} = \gamma F_n + o(1)$, this gives 
$$
\mu (\xifi) \ge \gamma^2, \quad \tmu (\xifi) \ge \frac{4 \gamma^2}{1 + \gamma^2} 
= \gamma^2 + \frac{1}{1 + \gamma^2}.
$$
By using triangle inequalities, it is easy to show that 
$$
\mu (\xifi) = \gamma^2,
$$
we omit the details. 
However, triangle inequalities are useless when we consider the multiplicative exponent. 

Let us now consider the point of view developed in Section \ref{sec:3} and
work with the rational numbers 
$$
w_m = w_{m, p} = {f_{m} \over p} + {f_{m-1} \over p^2} + \ldots + {f_1 \over p^m}, \quad m \ge 1. 
$$ 
It follows from Proposition \ref{fiborep} that,
for $n \ge 5$, there exist integers $\tir_n$ and $\tix_n$ such that 
$$
\biggl| w_{F_{n+2} - 2} - \frac{\tir_n}{p^{F_{n}} - 1} \biggr| \asymp_p p^{- 2 F_{n} - F_{n-1}}
$$
and
$$
\biggl| w_{2 F_{n+2} - 2} - \frac{\tix_n}{p^{F_{n}} - 1} \biggr| \asymp_p p^{- 3 F_{n} - F_{n-1}}.
$$
This implies, respectively, that $w_{F_{n+2} - 2}$ has a partial quotient $\asymp_p p^{F_{n-1}}$ and that 
$w_{2 F_{n+2} - 2}$ has a partial quotient $\asymp_p p^{F_{n} + F_{n-1}} \asymp_p p^{F_{n+1}}$. 

We apply Proposition \ref{etabound}
and use the fact that $F_n = \gamma^n / \sqrt{5} + o(1)$. 

Let $u$ be an integer with $2 F_n - 2 \le u \le 2 F_{n+1} - 2$. 
By Proposition \ref{etabound}, if $u \le F_{n+2} - 2$, then the rational number $w_u$ has a partial 
quotient $p^{\eta_u u}$ with 
\begin{align*}
\eta_u & \ge \frac{2 \gamma^n (1 + 1/(2 \gamma)) - \gamma^{n+2} (1 - \gamma^{-3})}
{2 \gamma^n (1 + 1/(2 \gamma)) + \gamma^{n+2} (1 - \gamma^{-3})} + o(1) \\
& = \frac{2 \gamma + 1 - (\gamma^3 - 1)}{2 \gamma + 1 + (\gamma^3 - 1)} + o(1) 
= \frac{1}{4 \gamma + 1} + o(1). 
\end{align*}
By Proposition \ref{etabound}, if $u \ge F_{n+2} - 2$, then the rational number $w_u$ has a partial 
quotient $p^{\eta_u u}$ with 
\begin{align*}
\eta_u & \ge \frac{\gamma^{n+2} (1 + \gamma^{-3}) -2 \gamma^{n+1}(1 - 1/(2 \gamma)) }
{\gamma^{n+2} (1 + \gamma^{-3}) + 2 \gamma^{n+1}(1 - 1/(2 \gamma)) }+ o(1) \\
& = \frac{\gamma^3 +1 - (2 \gamma^2 - \gamma)}{\gamma^3 + 1 + (2 \gamma^2 - \gamma)}+ o(1) 
= \frac{\gamma}{3 \gamma^2 + 1} + o(1). 
\end{align*}
Since the second lower bound is larger than the first one, we deduce from 
Theorem \ref{equ} that
$$
\htmu (\xifi) \ge \frac{2}{1 - 1/(4 \gamma + 1)} =  \frac{8 \gamma  +2}{4 \gamma} = 2 + \frac{1}{2 \gamma}.
$$
This completes the proof of the theorem. 
\end{proof}

Presumably, both inequalities in Theorem \ref{fibobound} are indeed equalities. 
This is the case if we restrict our attention to the approximations $|b \xifi - a|_p$, with $|b| \le |a|$. 
However, we do not see how to handle the approximations $|b \xifi - a|_p$, with $|b| > |a|$. 
To establish the lower bounds in Theorem \ref{fibobound}, 
we have used the very good rational 
approximations to the rational numbers $w_m$ coming from the repetitions in the prefix of length 
$m$ of $\bff$. But this allows us only to 
find approximations of denominator at most $p^{m/2}$. 
A way to handle the approximations $|b \xifi - a|_p$ with $|b| > |a|$ would be to look at the 
expansion of $1 / \xifi$ and hope to see repetitions in it. But we have no information on it. 


Numerical experiments suggest that no $w_m$ has a very large partial quotient 
associated with a convergent of denominator greater than $p^{m/2}$, but this seems 
difficult to prove.

\section{Additional remarks}   \label{sec:8}

Let us briefly sketch how Theorem \ref{equ} can be used to 
reprove some of the results of \cite{BuSc21}.

Let $p$ be a prime number and $\bfc = (c_k)_{k \ge 0}$ be a sequence of elements 
of $\{0, 1, \ldots , p-1\}$. As in Section \ref{sec:3}, write
$$
x_m = \frac{c_0}{p^m} + \frac{c_1}{p^{m-1}} + \ldots  + \frac{c_{m-1}}{p}, \quad m \ge 1.
$$
Assume that $\eta(\bfc) > 0$. 
Let $m$ be such that $x_m$ has a partial quotient $\asymp p^{\eta m}$, for some positive $\eta$ 
close to $\eta(\bfc)$. 
Then, the other partial quotients of $x_m$ are $\ll p^{(1- \eta)m}$. 
By the analysis made in Section \ref{sec:3},
for any positive integer $h$ less than $\eta m$, the rational number $x_{m+h}$ 
has a partial quotient $\asymp p^{\eta m - h}$, while its other 
partial quotients are $\ll p^{(1- \eta)m + h}$. In particular, for 
$h = \lceil (2 \eta - 1) m / 2 \rceil$, all the partial 
quotients of $x_{m+h}$ are $\ll p^{m/2}$, giving that 
$$
\heta (\bfc) \le \frac{m}{2 (m+h)} \le \frac{1}{2\eta + 1}.
$$
We conclude that 
\begin{equation} \label{hetabound}
\heta (\bfc) \le  \frac{1}{2\eta (\bfc) + 1}. 
\end{equation}
By Theorem \ref{equ}, this gives
$$
1 - \frac{2}{\htmu (\xi)} \le \frac{\tmu(\xi)}{3 \tmu(\xi) - 4}, \quad
\hbox{that is,} \quad {\htmu (\xi)} \le  3 + \frac{2}{\tmu(\xi) - 2},
$$
and we recover the first inequality of \cite[Theorem 3.1]{BuSc21}. 

\medskip

We can also recover the upper bound given in \eqref{eq:ineq2}. 
In view of Theorem \ref{equ}, it is equivalent to prove that $\heta (\bfc) \le 1 / \sqrt{5}$
always hold. The strategy is the following. For $m \ge 1$, define $\eta_m$ in such a way that 
$p^{\eta_m m}$ is the largest partial quotient of $x_m$. 
We take a local maximum $\alpha$ of the function $m \mapsto \eta_m$ 
and we consider the previous local maximum and the next one. 

For simplicity, we ignore the integral parts. 
Let $\alpha, \beta, \gamma, \delta, \eps$ and $m$ be such that 

\smallskip

$*$ $x_{m}$ has partial quotients $r_m \asymp p^{\alpha m}$, 
$\asymp p^{\beta m}$, and $\asymp p^{\eps m}$;

$*$ $x_{(1 +  \gamma )m}$ has partial quotients $\asymp p^{(\alpha - \gamma) m}$,  
$\asymp p^{(\beta +  \gamma)m}$;

$*$ $x_{(1 -  \delta) m}$ has partial quotients $\asymp p^{(\alpha  - \delta)m}$, 
$\asymp p^{(\beta -  \delta) m}$, and $\asymp p^{(\delta + \eps) m}$.

\smallskip

We assume that $x_m$ has a maximal partial quotient $r_m \asymp p^{\alpha m}$ 
and that $\alpha$ is arbitrarily close to $\eta (\bfc)$. 
We discuss the size of the largest
partial quotients of 
$x_{(1 +  \gamma )m}$ and $x_{(1 -  \delta) m}$. Since $r_m$ is assumed to be maximal, they have 
partial quotients $\asymp r_m p^{-\gamma m} \asymp p^{(\alpha - \gamma) m}$ 
and $\asymp r_m p^{-\delta m} \asymp p^{(\alpha  - \delta)m}$, respectively. 

Clearly, we have $0 \le \eps \le 1 - \alpha - \beta$. 
To bound $\heta (\bfc)$ from above, we choose $\gamma$ in such a way that the two largest 
partial quotients of 
$x_{(1 +  \gamma) m}$ are equal, that is, 
such that $\alpha - \gamma = \beta + \gamma$ holds. Thus,  
we take $\gamma = (\alpha - \beta) / 2$ and get the bound
$$
\heta (\bfc) \le \frac{\alpha + \beta}{2 (1 + \gamma)} = \frac{\alpha + \beta}{2 + \alpha - \beta}. 
$$
Likewise, we choose $\delta$ in such a way that the two largest 
partial quotients of 
$x_{(1 - \delta )m}$ are equal, that is, 
such that $\alpha - \delta = \delta + \eps$. Thus, we take $\delta = (\alpha - \eps) / 2$
and get the bound
$$
\heta (\bfc) \le \frac{\alpha + \eps}{2 (1 - \delta)} = \frac{\alpha + \eps}{2 - \alpha + \eps}.
$$
Since $\beta \le 1 - \alpha - \eps$,  we  have established that 
$$
\heta (\bfc) \le \min\Bigl \{ \frac{\alpha + \eps}{2 - \alpha + \eps}, \frac{1 - \eps}{2 \alpha + 1 + \eps} \Bigr\}.
$$
The right-hand side quantity in the minimum is at most equal to $1 / \sqrt{5}$ 
when $\alpha$ is greater than or equal to $(\sqrt{5} - 1) / 2$. 
The left-hand side quantity in the minimum is greater than $1 / \sqrt{5}$ if 
$$
\eps > \frac{2 - (1 + \sqrt{5}) \alpha}{\sqrt{5} - 1}.
$$
Under this assumption, the right-hand side quantity in the minimum is
$$
 < \frac{\sqrt{5} - 3 +  (1 + \sqrt{5}) \alpha}{(2 \alpha + 1) (\sqrt{5} - 1) + 2 - (1 + \sqrt{5}) \alpha}
= \frac{\alpha + 2 - \sqrt{5}}{(2 - \sqrt{5}) \alpha +  1} \le \frac{1}{\sqrt{5}},
$$
for $\alpha \le (\sqrt{5} - 1) / 2$. 


This gives $\heta (\bfc) \le 1 / \sqrt{5}$ in all cases and, by Theorem \ref{equ}, we recover the upper bound
$$
\htmu (\xi) \le \frac{5 +  \sqrt{5}}{2}, 
$$
established in \cite[Theorem 3.1]{BuSc21}. 
Our approach also shows that if $\heta (\bfc) = 1 / \sqrt{5}$, then $\eta (\bfc) = (\sqrt{5} - 1) / 2$, 
which, by Theorem \ref{equ}, corresponds to $\tmu (\xi) = 3 + \sqrt{5}$, as in \cite{BuSc21}. 


\medskip

We conclude with an open question. 

\begin{problem}
Let $b\ge 2$ be an integer.
Do there exist a bounded sequence of integers $\bfc = (c_k)_{k \ge 0}$, 
an infinite set ${\mathcal M}$ of positive integers, and a positive integer $C$ such that 
all the rational numbers
$$
x_m = \frac{c_0}{b^m} + \frac{c_1}{b^{m-1}} + \ldots  + \frac{c_{m-1}}{b}, \quad m \in {\mathcal M}, 
$$
have their partial quotients bounded from above by $C$? 
\end{problem}

For a given prime number $p$, this is a weaker question than the existence of 
a $p$-adic number $\xi$ such that 
$$
\inf_{a, b \not= 0} \, |ab| \cdot \vert  b \xi - a  \vert_p > 0,
$$
which, by \cite[Theorem 3]{BaBu21}, is equivalent to the existence 
of a sequence of integers $(c_k)_{k \ge 0}$ in $\{0, 1, \ldots , p-1\}$ 
and a positive integer $C$ such that 
all the rational numbers
$$
x_m = \frac{c_0}{b^m} + \frac{c_1}{b^{m-1}} + \ldots  + \frac{c_{m-1}}{b}, \quad m \ge 1, 
$$
have their partial quotients bounded from above by $C$.


\end{document}